\documentclass[11pt]{amsart}
\usepackage{amscd}
\usepackage{amsfonts}
\usepackage{color}
  \usepackage{amsmath,amssymb,amsthm,latexsym}
\usepackage{amscd}

\usepackage{multicol}
\usepackage{enumerate}

\newtheorem{theorem}{Theorem}[section]

\newtheorem{corollary}[theorem]{Corollary}

\usepackage{amsmath}
\usepackage{amsfonts}
\usepackage{amsmath,amsthm}
\usepackage{amscd}
\usepackage[all]{xy}
\numberwithin{equation}{section}
\theoremstyle{remark}
\newtheorem{remark}[theorem]{Remark}
\newtheorem{example}[theorem]{\bf Example}
\newcommand{\R}{\mathbb{R}}

\newcommand{\dd}{\mathrm{d}}

\begin{document}

\title[On Minimal hypersurfaces with symmetries]{\bf{On embedded minimal hypersurfaces in $S^{n+1}$ with  symmetries }}
\author{Changping Wang, Peng Wang}

\address{College of Mathematics and Informatics, FJKLMAA, Fujian Normal University, Fuzhou 350117, P. R. China}
\email{cpwang@fjnu.edu.cn}
\address{College of Mathematics and Informatics, FJKLMAA, Fujian Normal University, Fuzhou 350117, P. R. China}
\email{pengwang@fjnu.edu.cn}
\thanks{CPW was partly supported by the Project 11831005  of NSFC. PW was partly supported by the Project 11971107 of NSFC. The authors are thankful to Prof. Haizhong Li and Prof. Zhenxiao Xie for value discussions. The authors are thankful to Prof. Haizhong Li for pointing out to us Perdomo's conjecture in Reference \cite{Per2}}
\begin{abstract} In this note, we  generalize a characterization of the Clifford torus due to Ros.
	Let $f:M\rightarrow S^{n+1}$ be an embedded  closed minimal hypersurface. Assume there are $(n+2)$ great hyperspheres of $S^{n+1}$ perpendicular to each other, such that $M$ is symmetric with respect to them. Let $S$ denote the square of the length of the second fundamental form of $f$ and  let $\bar S=\frac{1}{Vol(M)}\int_{M} S\dd M$ be the average of $S$. Then $\bar S\geq n$ with equality holding if and only if $f$ is the Clifford torus $C_{m,n-m}$. It can be rewritten as a Simons' type theorem: If $0\leq \int_M (n-S)\dd M$, then either $S\equiv0$ or $S\equiv n$. This answers partially   a conjecture by Perdomo.
  Moreover, the estimate of the Willmore energy of $f$ is built: $W(M)\geq n^{\frac{n}{2}}Vol(M)$.
\end{abstract}

\date{\today}
\maketitle

\vspace{0.5mm}  {\bf \ \ ~~Keywords:}    Clifford torus; Average of length of the Second fundamental form; generalized Lawson conjecture; Chern's Conjecture; Willmore energy.\vspace{2mm}

{\bf\   ~~ MSC(2020): \hspace{2mm} 53C40, 53A31}


\section{Introduction}

The study of minimal hypersurfaces in $S^{n+1}$ is an important topic in geometry and leads to many progress in geometric analysis and other fields. The famous Yau conjecture states that the embeddness of the (oriented and closed) minimal hypersurface will ensure that first eigenvalue (of its Laplacian) is equal to $n$. It attains a wide interest and stay an open problem so far, except some partial results. See e.g. \cite{choesoret,Ros,TY2} and reference therein for some partial results and further discussion. If Yau's conjecture holds, then the Courant Nodal Theorem will show that every hypersphere of $S^{n+1}$ will divide the minimal hypersurface into two connected components, since now the coordinate functions are first eigenfunctions. In \cite{Ros}, Ros proved that such a two-piece property does hold for the case of $n=2$. Moreover, he proved the following description of the Clifford torus (which now is a corollary of \cite{Brendle}):
\begin{theorem}{\em Ros }\cite{Ros} \label{thm-0} Assume that the embedded minimal torus $f$ in $S^{3}$ is symmetric with respect to four coordinate hyperplanes, then $f$ is congruent to the Clifford torus.
\end{theorem}
Ros' two-piece property theorem  was generalized to all oriented, embedded, closed minimal hypersurfaces by Choe and Soret \cite{choesoret} (an alternate proof can also be found in \cite{McG}). It is therefore natural to ask whether a similar characterization of the Clifford tori in $S^{n+1}$ holds or not, which turns out to be the main topic of this paper.\ \vspace{3mm}

Throughout this paper, we assume $M$ to be an oriented closed $n-$dimensional manifold.
Let $f:M\rightarrow S^{n+1}$ be a hypersurface in the unit sphere $S^{n+1}$, with its first and second fundamental form being $I=\sum_i(\omega^i)^2$ and $II=\sum_{i,j}h_{ij}\omega^i\omega^j$. The mean curvature $H$, the square of the length of the second fundamental form $S$ and
 the length of the traceless second fundamental form $\rho$ of $f$ are of the form
\begin{equation}
H=\frac{1}{n}\sum_{i}h_{ii},\ S=\sum_{i,j}(h_{ij})^2,\ \rho=\sqrt{S-nH^2}.
\end{equation}
We denote by $\bar S$ the average of $S$:
\begin{equation}
\bar S:=\frac{1}{Vol(M)}\int_MS \dd M.
\end{equation}
Here $Vol(M)=\int_M \dd M$ denotes the volume of $f$.
The Willmore functional is defined to be
\begin{equation}
W(M):=\int_M\rho^n\dd M.
\end{equation}
It is well-known that it is conformally invariant \cite{Chen, Ko, Li-h, Li-s, Wang}. We recall that the Clifford minimal tori are (see for example \cite{GLW, Li-h, Li-s} for more details.)
		\begin{equation}
			C_{m,n-m}=S^m\left(\sqrt{\frac{m}{n}}\right)\times S^{n-m}\left(\sqrt{\frac{n-m}{n}}\right), \ 1\leq m\leq n-1.
		\end{equation}

 Since we can chose $(n+2)$ coordinate hyperplanes such that the $(n+2)$ great hyperspheres perpendicular to each other are the intersections of $S^{n+1}$ with them, we will simply assume the hypersurfaces symmetric under the $(n+2)$ coordinate hyperplanes reflections.
\begin{theorem}\label{thm-1}
	Let $f:M\rightarrow S^{n+1}$ be an embedded  closed minimal hypersurface symmetric under the $(n+2)$ coordinate hyperplanes reflections, which is  not totally geodesic. Then
	\begin{equation}
	\bar S\geq n,
	\end{equation}
	with equality holding if and only if $f$ is congruent to   some Clifford minimal torus $C_{m,n-m}$.
\end{theorem}

We can also rewrite it as a Simons' type Theorem
\begin{theorem}\label{thm-11}
	Let $f:M\rightarrow S^{n+1}$ be an embedded  closed minimal hypersurface symmetric under the $(n+2)$ coordinate hyperplanes reflections. If
	\begin{equation}
	0\leq \int_M (n-S)\dd M,
	\end{equation}
then either $S\equiv 0$ and $f$ is a totally geodesic hypersphere, or $S\equiv n$ and $f$ is congruent to   some Clifford minimal torus $C_{m,n-m}$.
\end{theorem}

\begin{remark}\
\begin{enumerate}
\item Theorem \ref{thm-1} and Theorem \ref{thm-11} have close relations with the famous characterizations of rigidity of minimal hypersurfaces due to J. Simons \cite{Simons}, Chern-de Carmo-Kobayashi \cite{CDK} and Lawson \cite{Lawson}. The Simons inequality shows that if $0\leq S\leq n$, then $S\equiv0$ or $S\equiv n$. Moreover, if $S\equiv0$, we obtain round hypersphere. If $S\equiv n$, one obtains  the Clifford minimal tori $C_{m,n-m}$, $m=1,\cdots, n-1$ \cite{CDK,Lawson}. In the proof of the above results one will see that for  embedded  closed minimal hypersurface with given symmetries, the average $\bar S= n$ will force $S\equiv n$ and hence one obtains  the Clifford minimal tori $C_{m,n-m}$, $m=1,\cdots, n-1$.
\item Motivated by the work of \cite{Simons,CDK,Lawson}, Chern proposed the famous Chern conjecture, asking if the value distribution of $S$ is discrete or not, when $S$ is constant.  It turns out to be an important topic and received a lot of interesting results, some of which are also related to the Willmore functionals of hypersurfaces (See \cite{Chen,Ko,Li-h,Li-s,Marques2,Pinkall,Wang} for discussions on Willmore hypersurfaces). We refer to \cite{DX,PT,TY2,TY3,WX,YC} and reference therein for recent progress on this direction.
    \item Theorem \ref{thm-1} shows that the following open problem  due to Perdomo \cite{Per2} holds under the symmetric assumptions:

     Let $f$ be an oriented,  embedded,  closed minimal hypersurface with the average $\bar S= n$, i.e.,  $\int_M S\dd M=nVol(M).$
       Then $f$ must be congruent to some Clifford minimal torus $C_{m,n-m}$.

Perdomo proved that the above conjecture holds if $Index(M)=n+2$ \cite{Per2}, or if $M$ has two different principle curvatures \cite{Per3}.
       As pointed out in \cite{Per2}, when $n=2$, this goes back to the famous Lawson conjecture for embedded torus, which was proved by Brendle \cite{Brendle}. So this can be viewed as  the hypersurface version of the Lawson conjecture (note also that when $S\equiv n$ it reduces to the rigidity theorem of Chern-de Carmo-Kobayashi \cite{CDK} and Lawson \cite{Lawson}).

       Different from the case of surfaces, there are many embedded minimal hypersurfaces  in $S^{n+1}$ different from the great hyperspheres, which are diffeomorphic to $S^n$ or $S^1\times S^{n-1}$, when $n\geq 3$, \cite{Hsiang1,Hsiang2,Hsiang-S,Hsiang3}. So it is a more subtle problem for embedded minimal hypersurfaces  in $S^{n+1}$.  Perdomo's result \cite{Per3}
might be helpful for the discussions of the examples in \cite{Hsiang1,Hsiang2,Hsiang-S,Hsiang3}.
        \item One may consider furthermore the generalized Chern conjecture or Chern problem, that is, the value distribution of $\bar S$. Set
         $ \mathfrak{S}_a$ to be the set of all possible $\bar S$ for all  embedded closed minimal hypersurfaces $M$ in $ S^{n+1}$
        and set $ \mathfrak{S}_0$ to be the set of all possible numbers equal to $S$ of some embedded closed minimal hypersurface $M$ in $ S^{n+1}$  with constant $S$.   Chern's conjecture implies that $\mathfrak{S}$ is also a discrete set, with $\mathfrak{S}\subset \mathfrak{S}_a$. So it is natural to ask whether $\mathfrak{S}_a$ is also a discrete set or not.

\item In view of the above theorem, Simons' Theorem \cite{Simons} and Perdomo's Theorem \cite{Per2}, for general case it seems  reasonable to expect a gap phenomenon for $\bar S$: there exists some constant $C_n\leq n$ depending only on $n$ such that if $0\leq \bar S<C_n$, then $f$ is totally geodesic.

 \item
 The normalized scalar curvature satisfies  $R=1-\frac{1}{n(n-1)}S$ (See for example (2.6) of \cite{Li-s}, \cite{Per2}).
           So the above problem is equivalent to asking whether the average total scalar curvature of  embedded  closed minimal hypersurfaces in $S^{n+1}$ are quantized or not, although there are no topological obstructions when $n\geq3$.
          \item For CMC hypersurfaces in $S^{n+1}$, one can consider the similar questions in view of Pinkall-Sterling Conjecture \cite{AL} and its proof \cite{PT}.
\end{enumerate}
\end{remark}

Theorem \ref{thm-1} also yields an estimate of the Willmore functional of minimal hypersurfaces.
\begin{theorem}\label{thm-2}
	Let $f:M\rightarrow S^{n+1}$ be an embedded  closed minimal hypersurface symmetric under the $n+2$ coordinate hyperplances reflections.  Assume that $f$ is  not totally geodesic.  Then
	\begin{equation}\label{eq-W0}
	W(M)\geq n^{\frac{n}{2}} Vol(M),
	\end{equation}
	with equality holding if and only if  $f$ is congruent to   some Clifford minimal torus $C_{m,n-m}$.
\end{theorem}
Note that  the Clifford minimal torus $C_{m,n-m}$ is   Willmore   if and only if $n=2m$ \cite{GLW}.\ \\
	
We will give the proofs of Theorem \ref{thm-1} and Theorem \ref{thm-2} in Section 2 and Section 3 respectively.

\section{On minimal hypersurfaces with symmetries}

First we recall the two-piece properties of embedded minimal hypersurfaces in $S^{n+1}$ due to Ros \cite{Ros} when $n=2$ and Choe-Soret \cite{choesoret} when $n\geq 3$ (See   \cite{McG} for an alternate proof).
\begin{theorem} {\em Ros, Choe and Soret } \cite{Ros,choesoret}\label{thm-t} Let $f:M\rightarrow S^{n+1}\subset\R^{n+2}$ be an embedded, oriented minimal hypersurface, which is not totally geodesic. Then any great hypersphere of $S^{n+1}$ divides $f$ into two connected pieces.
\end{theorem}

The proof of Theorem 1.1 by Ros \cite{Ros}  can be used directly to give a proof of Theorem \ref{thm-1}. For readers' convenience, we include it here.

\begin{proof} of Theorem \ref{thm-1}:  Let $f_i$ denote the $i-th$
 coordinate function of the position vector $f$ and let $N_i$ denote the $i-th$ coordinate function of the unit normal vector $N$ of $f$, $1\leq i\leq n+2$. Consider the coordinate hyperplane $P_i:=\{x_i=0\}$, $1\leq i\leq n+2$. By Theorem \ref{thm-t}, it divides $M$ into two connected domains. Let $M_i^+=M\cap\{x_i\geq0\}$ and  $M_i^-=M\cap\{x_i\leq0\}$. Since $f$ is minimal, we obtain
 \begin{equation}\label{eq-min}
 \Delta f_i+nf_i=0,\ \Delta N_i+ S N_i=0.
 \end{equation}

Since $f_i$ does not change sign on $M_i^+$, we conclude by the  Courant Nodal Theorem that the first eigenvalue of the Laplacian on $M_i^+$, with the Dirichlet boundary condition, is equal to $n$. Since $f$ is symmetric with the coordinate hyperplane $P_i$ and since $f$ is not contained in $P_i$ (otherwise $f$ is totally geodesic), we have that $f$ insects $P_i$ orthogonally, that is, the unit normal $\mathbf {e}_i=(0,\cdots0,1,0,\cdots,0)^t$ of $P_i$ is contained in $T_pM$ for each $p\in M\cap P_i=\partial M_i^+=\partial M_i^-$. So $N\perp \mathbf e_i$ on $\partial M_i^+$, that is,
\[N_i|_{\partial M_i^+}=0.\]
Therefore
\begin{equation}\label{eq-Ni}
\int_{M_{i}^+}  S N_i^2\dd M=-\int_{M_{i}^+}  N_i\Delta N_i\dd M=\int_{M_{i}^+} |\nabla N_i|^2\dd M\geq
 n\int_{M_{i}^+} N_i^2\dd M,
\end{equation}
since $n$ is the first eigenvalue of the Laplacian on $M_i^+$.

Similarly we have
\begin{equation*}
\int_{M_{i}^-}  S N_i^2\dd M \geq
 n\int_{M_{i}^-} N_i^2\dd M,
\end{equation*}
and hence
\begin{equation}
\int_{M}  S N_i^2\dd M \geq
 n\int_{M} N_i^2\dd M.
\end{equation}
Summing up, we obtain
\begin{equation}
\int_{M}  S \dd M=\sum_i\int_{M} S N_i^2\dd M \geq
 n\int_{M}\dd M.
\end{equation}
If the equality holds, then all the inequalities appeared above become equalities. In particular, by \eqref{eq-Ni} we see that $N_i$ is also an eigenfunction with eigenvalue $n$.
As a consequence, we obtain $S\equiv n$. By \cite{CDK,Lawson1}, $f$ is congruent to some Clifford minimal torus $C_{m,n-m}$.
\end{proof}
Note that in the above proof, the embedness is only used to derive the two-piece properties. So  if we replace the embedding assumption by  the first eigenvalue $\lambda_1$ of the Laplacian of $f$ being equal to $n$, then any coordinate hyperplane still divides $M$ into two conneced pieces. The other proofs are the same and hence we obtain the following theorem.
\begin{theorem}\label{thm-4}
	Let $f:M\rightarrow S^{n+1}$ be an oriented, non-totally geodesic, closed minimal hypersurface, which  is  symmetric under the $(n+2)$ coordinate hyperplanes reflections. Assume that  $\lambda_1=n$, where $\lambda_1$ denotes the first eigenvalue of the Laplacian of $f$. Then
	\begin{equation}
	\bar S\geq n
	\end{equation}
	with equality holding if and only if $f$ is congruent to some  Clifford torus $Cl_{m,n-m}$.
\end{theorem}
We would like to have more discussions on Perdomo's interesting relevant works \cite{Per1,Per2}.
\begin{remark} \
\begin{enumerate}
	\item In \cite{Per2}, Perdomo shows that if $f$ is embedded and has Morse index $n+2$ (see \cite{Simons,Per1,Per2} for the discussion on index of minimal submanifolds), then $\bar S\geq n$ holds and equiliy holds if and only if $f$ is  $f$ is congruent to some  Clifford torus $Cl_{m,n-m}$, which generalized the famous characterization of Clifford torus \cite{Ur}.
	The estimate of index of Clifford torus plays an important role in the proof of Willmore conjecture in $S^3$ \cite{Marques}.
	\item  Another interesting estimate due to Perdomo \cite{Per1} is the following characterization of minimal hypersurfaces with symmetries.
\end{enumerate}	
\end{remark}
We define the symmetric subgroup of $M$ in $O(n+2)$ as follow
\[O_M(n+2):=\{\gamma\in O(n+2)| \gamma (M)=M\}.\]
\begin{theorem} {\em Perdomo} \cite{Per1} \label{thm-P}
		Let $f:M\rightarrow S^{n+1}$ be an oriented,  closed, non-equatorial minimal hypersurface.
	Assume that $O_M(n+2)$ only fixes the original points of $\R^{n+2}$.	Then
	$Index(M)\geq n+2$,		with equality holding if and only if $f$ is congruent to  some Clifford minimal torus $C_{m,n-m}$.
\end{theorem}
This gives an index estimate of minimal hypersurfaces symmetric under all coordinate hyperplanes reflections, since the combination of them contains the anti-podal symmetry, which only fixes the original points of $\R^{n+2}$.
\begin{corollary}
	Let $f:M\rightarrow S^{n+1}$ be an oriented,  closed, non-equatorial minimal hypersurface, symmetric under all coordinate hyperplanes reflections. Then 	$Index(M)\geq n+2$,		with equality holding if and only if $f$ is congruent to  some Clifford minimal torus $C_{m,n-m}$.
\end{corollary}
\vspace{2mm}

\section{Estimate of Willmore energy of minimal hypersurfaces with symmetries}

The study of the relations between the Willmore functional and other geometric data of a hypersurfaces or in general a submanifold are highly non-trivial when the dimension of $M$ is greater or equal to $3$. We first use the classical H\"{o}lder inequality to obtain a simple estimate of Willmore functional  for all $n-$dimensional submanifolds in $S^{n+p}$, which holds trivially when $n=2$. We refer to \cite{Li-s,Wang} for more details on Willmore submanifolds and \cite{Li-y,Marques,Marques2, Montiel} for further discussions on Willmore surfaces.
\begin{theorem}\label{thm-22}
	Let $f:M\rightarrow S^{n+p}$ be an $n-$dimensional, oriented closed submanifold. Then
	\begin{equation}\label{eq-W1}
	W(M)\geq  \frac{1}{Vol(M)^{\frac{n}{2}-1}}\left(\int_M(S-nH^2)\dd M\right)^{\frac{n}{2}}=\left(\bar S-n\overline{(H^2)}\right)^{\frac{n}{2}}Vol(M).
	\end{equation}
	Here we define
$\overline{(H^2)}:=\frac{1}{Vol(M)}\int_MH^2\dd M.$
	In particular, if $f$ is minimal, then
		\begin{equation}\label{eq-W2}
		W(M)\geq \bar S^{\frac{n}{2}}Vol(M).
		\end{equation}
\end{theorem}

\begin{proof}
By the
 H\"{o}lder inequality, we obtain
 \[ \left(\int_M \rho^n\dd M\right)^{\frac{2}{n}}\left(\int_m1^{\frac{n}{2}}\dd M\right)^{\frac{n-2}{n}}\geq \int_M \rho^2\dd M,
 \]
i.e.
\[W(M)^{\frac{2}{n}}\left(Vol(M)\right)^{\frac{n-2}{n}}\geq \int_M (S-nH^2)\dd M.\]
Hence \eqref{eq-W1} follows. Substituting $H\equiv0$ into  \eqref{eq-W1} for minimal submanifolds yields  \eqref{eq-W2}.
\end{proof}

Theorem \ref{thm-2} follows immediately from Theorem \ref{thm-22}.
\begin{proof} of Theorem \ref{thm-2}: By Theorem \ref{thm-1}, we have
\[\bar S\geq n\]
and equality holds if and only $f$ is congruent to the Clifford torus $C_{m,n-m}$ for some $m\in\{1,2,\cdots,n-1\}$.  Substituting this into \eqref{eq-W2}, we get \eqref{eq-W0}.
If the equality holds, we see that first one needs $\bar S=n$, which means $f$ is congruent to the Clifford minimal tori $C_{m,n-m}$, $m=1,\cdots, n-1$. For  $C_{m,n-m}$, one check easily the equality case of \eqref{eq-W0} holds.
\end{proof}

Similar to Theorem \ref{thm-4}, we also have
\begin{theorem}\label{thm-5}
	Let $f:M\rightarrow S^{n+1}$ be an closed minimal hypersurface symmetric under the $n+2$ coordinate hyperplances reflections.  Assume that $f$ is  not totally geodesic and $\lambda_1=n$, where $\lambda_1$ denotes the first eigenvalue of the Laplacian of $f$. Then
	\begin{equation}\label{eq-W0}
	W(M)\geq n^{\frac{n}{2}} Vol(M),
	\end{equation}
	with equality holding if and only if  $f$ is congruent to   some Clifford minimal torus $C_{m,n-m}$.
\end{theorem}

  \def\refname{References}
  
  	\vspace{2mm}

  	\end{document}